\documentclass[11pt,a4paper]{article}

\usepackage{authblk}

\usepackage[T1]{fontenc}    
\usepackage{enumitem}
\usepackage{mathrsfs}
\usepackage{amsfonts}
\usepackage{amsthm}
\usepackage{authblk}
\usepackage{cite}
\usepackage{amsmath,amssymb,amscd,amsthm}
\usepackage{graphics}
\input xy
\xyoption{all}
\usepackage[colorlinks=true, linkcolor=magenta, urlcolor=blue, citecolor=blue]{hyperref}
\numberwithin{equation}{section}

\usepackage{scalefnt}   

\usepackage{comment}
\usepackage{setspace}

\usepackage[paperwidth=210mm,
paperheight=297mm,
a4paper,
twoside,
top=0.8in, left=0.9in, right=0.9in, bottom=1in]{geometry}

\usepackage[active]{srcltx}

\theoremstyle{definition}
\newtheorem{lemma}{Lemma}[section]

\newtheorem{defi}[lemma]{Definition}
\newtheorem{thm}[lemma]{Theorem}

\newtheorem{rmk}[lemma]{Remark}

\newcommand{\C}{\mathbb{C}}
\newcommand{\N}{\mathbb{N}}
\newcommand{\Z}{\mathbb{Z}}

\usepackage{authblk} 
\usepackage{hyperref}
\usepackage{float}
\restylefloat{figure}

\allowdisplaybreaks

\begin{document}
	\title{Simple modules of the planar Galilean conformal algebra from tensor products}

    \author[1]{Dashu Xu\thanks{E-mail: dox@mail.ustc.edu.cn}}


	\date{}
	\maketitle
	\begin{abstract}
		This paper is devoted to constructing simple modules of the planar Galilean conformal algebra.
		We study the tensor products of finitely many simple $\mathcal{U}(\mathcal{H})$-free modules with an arbitrary simple restricted module, where $\mathcal{H}$ is the Cartan subalgebra.
		We establish necessary and sufficient conditions for simplicity and determine the corresponding isomorphism classes.

	\bigskip
		
	\end{abstract}
	
	\section{Introduction}
	The planar Galilean conformal algebra is an infinite-dimensional Lie algebra arising in the field of theoretical physics.
	It was introduced by Bagchi and Gopakumar in \cite{BG} to study a non-relativistic conformal symmetry.
	Of particular algebraic interest are the structure and representation theory of the planar Galilean conformal algebra.
	Various types of derivations, automorphism groups, left symmetric algebra structures and universal central extensions of the planar Galilean conformal algebra have been explicitly determined in a series of papers\cite{CH,CSY,GLP,TZ,CS,JSL}.  
	Beyond the structural perspective, representation theory concentrates on studying different types of modules over the planar Galilean conformal algebra.
	Weight modules are studied in \cite{A,AK}, where the Kac determinant for a highest weight Verma module is given explicitly. 
	Under the framework of $(d,\sigma)$-twisted affine-Virasoro superalgebras, simple quasi-finite modules are classified in \cite{LYZ}.
	Regarding non-weight representations, simple Whittaker and $\mathcal{U}(\mathcal{H})$-free modules are classified in \cite{CYY,CGZ1}.
	The planar Galilean conformal algebra admits a $\Z$-gradation, which induces the notion of restricted modules.
	In the spirit of \cite{MZ}, the authors of \cite{GG,CY} constructed a large family of simple restricted modules and characterize them conceptually.
	However, the problem of completely classifying simple restricted modules has not yet been solved.
	
	The tensor product construction provides an effective manner for expanding simple modules from known ones of a given Lie algebra.
	For example, the tensor products of simple Virasoro algebra modules generate numerous new simple modules, see \cite{GLZ,CGZ,LGW,OW,TanZ} and the references therein.
	In this paper, we study the tensor products of finitely many simple $\mathcal{U}(\mathcal{H})$-free modules and an arbitrary simple restricted module.
	This work is motivated by previous studies on the representation theory of the Virasoro algebra\cite{TanZ,TZhao} along with investigations into other related algebras\cite{CGZ2,CY1}.
	We give necessary and sufficient conditions for such a tensor product module to be simple, and subsequently determine their isomorphism classes.
	The \textbf{main results} of this paper are presented in Theorems \ref{simple} and \ref{iso}.
	
	Throughout this paper, we use $\N$, $\Z$, $\C$ and $\C^*$ to denote the sets of non-negative integers, integers, complex numbers and non-zero complex numbers respectively.
	Additionally, for any arbitrary Lie algebra $\mathfrak{g}$, we employ $\mathcal{U}(\mathfrak{g})$ to denote its universal enveloping algebra.
\section{Preliminaries}

    First, let us recall the definition of the planar Galilean conformal algebra.
    \begin{defi}
    	The planar Galilean conformal algebra, denoted by $\mathcal{G}$, is an infinite-dimensional Lie algebra, which has a $\C$-basis $\left\{L_n,H_n,I_n,J_n\mid n\in\Z\right\}$ with the following non-trivial Lie brackets:
    	\begin{align*}
    		[L_m,L_n]&=(n-m)L_{m+n},\quad [L_m,H_n]=nH_{m+n},\quad [L_m,I_n]=(n-m)I_{m+n},\\
    		[L_m,J_n]&=(n-m)J_{m+n},\quad [H_m,I_n]=I_{m+n},\quad [H_m,J_n]=-J_{m+n}. 
    	\end{align*}
    \end{defi}
    For $n\in\Z$, denote by $\mathcal{G}_n=\C L_n\oplus\C H_n\oplus\C I_n\oplus\C J_n$.
    Then, we have a decomposition
    \[
    \mathcal{G}=\bigoplus_{n\in\Z}\mathcal{G}_n,
    \]
    and for arbitrary $m,n\in\Z$, the commutator $[\mathcal{G}_m,\mathcal{G}_n]$ is contained in $\mathcal{G}_{m+n}$.
    It follows that $\mathcal{G}$ is a $\Z$-graded Lie algebra.
    \begin{defi}
    	A $\mathcal{G}$-module $V$ is called a restricted module provided for any $v\in V$, there exists an integer $i\in\N$ such that $\mathcal{G}_nv=0$ for $n\ge i$.
    \end{defi}
    
    For $\eta\in\C$ and $\lambda,\sigma\in\C^*$, denote by $\Omega(\lambda,\sigma,\eta)=\C[s,t]$ and $\Gamma(\lambda,\sigma,\eta)=\C[x,y]$, respectively. 
    Then, by Definition 4.3 in \cite{CGZ1}, $\Omega(\lambda,\sigma,\eta)$ and $\Gamma(\lambda,\sigma,\eta)$ admit the structures of $\mathcal{G}$-modules under the following actions:
    \begin{equation}
    	\begin{split}
    		\Omega(\lambda,\sigma,\eta)=\C[s,t]:\quad L_nf(s,t)&=\lambda^n(s+nt+n\eta)f(s-n,t),\\
    		H_nf(s,t)&=\lambda^n tf(s-n,t),\\
    		I_nf(s,t)&=\lambda^n\sigma f(s-n,t-1),\\
    		J_nf(s,t)&=0,
    	\end{split}
    \end{equation}
    where $f(s,t)\in \C[s,t]$,
    \begin{equation}
    	\begin{split}
    		\Gamma(\lambda,\sigma,\eta)=\C[x,y]:\quad L_ng(x,y)&=\lambda^n(x+ny+n\eta)g(x-n,y),\\
    		H_ng(x,y)&=\lambda^n yg(x-n,y),\\
    		I_ng(x,y)&=0,\\
    		J_ng(x,y)&=\lambda^n\sigma g(x-n,t+1),
    	\end{split}
    \end{equation}
    where $g(x,y)\in\C[x,y]$.
    
    Denote by $\mathcal{H}=\C L_0\oplus\C H_0$.
    Combining Theorems 4.12, 4.14 and 4.15 in \cite{CGZ1}, the following theorem is valid.
    \begin{thm}
     Suppose $M$ is a simple $\mathcal{G}$-module 
    that is free of rank one when restricted to $\mathcal{U}(\mathcal{H})$.
    Then $M$ is either isomorphic to $\Omega(\lambda,\sigma,\eta)$ or $\Gamma(\lambda,\sigma,\eta)$.
	\end{thm}
	Additionally, we need the following facts in the field of Lie algebras.
	
	For an arbitrary Lie algebra $\mathfrak{g}$, suppose $\left\{M_i\mid 1\le i\le n\right\}$ is a family of $\mathfrak{g}$-modules.
	Then the tensor product of vector spaces
	\[
	M_1\otimes\cdots\otimes M_n
	\]
	is a $\mathfrak{g}$-module, where for $x\in\mathfrak{g}$ and $m_i\in M_i$, the $\mathfrak{g}$-actions are given by
	\[
	x(m_1\otimes\cdots\otimes m_n)=\sum_{i=1}^{n}m_1\otimes\cdots\otimes(xm_i)\otimes\cdots\otimes m_n.
	\] 
	It is obvious that $M_1\otimes\cdots\otimes M_n\cong M_{\sigma(1)}\otimes\cdots\otimes M_{\sigma(n)}$ as $\mathfrak{g}$-modules, where $\left\{\sigma(1),\sigma(2),\ldots,\sigma(n)\right\}$ is a permutation of $\left\{1,2,\ldots,n\right\}$.
	
	A crucial lemma from linear algebra has been proved in \cite{TanZ}.
	\begin{lemma}[Lemma 2 in \cite{TanZ}]
		\label{crucial}
		Suppose $\lambda_1,\ldots,\lambda_m\in\C^*$, $s_1,\ldots,s_m\in\Z_{\ge1}$ and $s_1+\cdots+s_m=s$.
		For $n\in\Z$, $1\le t\le m$ and $s_1+\cdots+s_{t-1}+1\le k\le s_1+\cdots+s_t$, define $f_k(n)=n^{k-1-\sum_{j=1}^{t-1}s_j}\lambda_t^n$. 
		Let $\mathfrak{R}=(y_{pq})$ be the $s\times s$ matrix with $y_{pq}=f_q(p+r-1)$, where $r\ge0$ and $p, q=1,2,\ldots,s$.
		Then we have
		\begin{equation*}
			\det\mathfrak{R}=\prod_{j=1}^{m}(s_j-1)^{!!}\lambda_j^{{s_j(s_j+2r-1)}/{2}} \prod_{1\le i<j\le m}(\lambda_j-\lambda_i)^{s_is_j},
		\end{equation*}
		where $m^{!!}=m!\times(m-1)!\times\cdots\times 1!$ for $m\in\Z_{\ge 1}$ and $0^{!!}=1$.
	\end{lemma}
\section{Simple Tensor Product Module}
For $\boldsymbol{m}=(m_1,m_2)\in\N^2$, $\boldsymbol{\lambda}=(\lambda_1,\lambda_2,\ldots,\lambda_{m_1+m_2}),\,\boldsymbol{\sigma}=(\sigma_1,\sigma_2,\ldots,\sigma_{m_1+m_2})\in(\C^*)^{m_1+m_2}$, $\boldsymbol{\eta}=(\eta_1,\eta_2,\ldots,\eta_{m_1+m_2})\in\C^{m_1+m_2}$, and an arbitrary simple restricted $\mathcal{G}$-module $V$, define the tensor product $\mathcal{G}$-module as follows:
\begin{equation}
	\label{tp}
	\boldsymbol{T}(\boldsymbol{m},\boldsymbol{\lambda},\boldsymbol{\sigma},\boldsymbol{\eta},V)=\left(\bigotimes_{k=1}^{m_1}\Omega(\lambda_k,\sigma_k,\eta_k)\right)\otimes\left(\bigotimes_{k=m_1+1}^{m_1+m_2}\Gamma(\lambda_k,\sigma_k,\eta_k)\right)\otimes V.
\end{equation}
As vector spaces, we have 
\[
\Omega(\lambda_k,\sigma_k,\eta_k)=\C[s_k,t_k]
\] 
for $1\le k\le m_1$ and \[\Gamma(\lambda_k,\sigma_k,\eta_k)=\C[x_{k-m_1},y_{k-m_1}]\]
for $m_1+1\le k\le m_1+m_2$.
Hence, there is a linear isomorphism
\[
\boldsymbol{T}(\boldsymbol{m},\boldsymbol{\lambda},\boldsymbol{\sigma},\boldsymbol{\eta},V)\cong \C[s_1,\ldots,s_{m_1},x_1,\ldots,x_{m_2},t_1,\ldots,t_{m_1},y_1,\ldots,y_{m_2}]\otimes V.
\]
In the following, for any $v\in V$, we use $\boldsymbol{s}^{\boldsymbol{p}}\boldsymbol{t}^{\boldsymbol{q}}\otimes\boldsymbol{x}^{\boldsymbol{r}}\boldsymbol{y}^{\boldsymbol{s}}\otimes v$ to denote $$s_1^{p_1}t_1^{q_1}\otimes\cdots\otimes s_{m_1}^{p_{m_1}}t_{m_1}^{q_{m_1}}\otimes x_1^{r_1}y_1^{s_1}\otimes\cdots\otimes x_{m_2}^{r_{m_2}}y_{m_2}^{s_{m_2}}\otimes v$$ for brevity, where
\begin{align*}
\boldsymbol{p}&=(p_1,\ldots,p_{m_1}),\quad\boldsymbol{q}=(q_1,\ldots,q_{m_1})\in\N^{m_1},\quad
\boldsymbol{r}=(r_1,\ldots,r_{m_2}),\quad\boldsymbol{s}=(s_1,\ldots,s_{m_2})\in\N^{m_2}.
\end{align*}

For $m\in\N$, recall that there is a \textbf{lexicographical order} on $\N^m$, which is defined by
\[
(a_1,\ldots,a_m)\succ(b_1,\ldots,b_m)\Leftrightarrow \exists\; 1\le k\le m\quad\mbox{s.t.}\quad a_1=b_1,\ldots a_{k-1}=b_{k-1}\quad\mbox{and}\quad a_k>b_k.  
\]
Now, for $\boldsymbol{p}=(p_1,\ldots,p_{m_1}),\boldsymbol{q}=(q_1,\ldots,q_{m_1})\in\N^{m_1}$ and 
$\boldsymbol{r}=(r_1,\ldots,r_{m_2}),\boldsymbol{s}=(s_1,\ldots,s_{m_2})\in\N^{m_2}$, define
\[
(\boldsymbol{p},\boldsymbol{q},\boldsymbol{r},\boldsymbol{s})\succ  
(\boldsymbol{a},\boldsymbol{b},\boldsymbol{c},\boldsymbol{d})
\]
if and only if
\[
\boldsymbol{p}\succ\boldsymbol{a}\quad\mbox{or}\quad\boldsymbol{p}=\boldsymbol{a}, \boldsymbol{r}\succ\boldsymbol{c}\quad\mbox{or}\quad\boldsymbol{p}=\boldsymbol{a}, \boldsymbol{r}=\boldsymbol{c},\boldsymbol{q}\succ\boldsymbol{b}\quad\mbox{or}\quad\boldsymbol{p}=\boldsymbol{a}, \boldsymbol{r}=\boldsymbol{c},\boldsymbol{q}=\boldsymbol{b}, \boldsymbol{s}\succ\boldsymbol{d}.
\]
For a non-zero element
\[
g=\sum_{(\boldsymbol{p},\boldsymbol{q},\boldsymbol{r},\boldsymbol{s})\in E}\boldsymbol{s}^{\boldsymbol{p}}\boldsymbol{t}^{\boldsymbol{q}}\otimes\boldsymbol{x}^{\boldsymbol{r}}\boldsymbol{y}^{\boldsymbol{s}}\otimes v_{\boldsymbol{p}\boldsymbol{q}\boldsymbol{r}\boldsymbol{s}}\in\boldsymbol{T}(\boldsymbol{m},\boldsymbol{\lambda},\boldsymbol{\sigma},\boldsymbol{\eta},V),
\]
where $E\subset(\N^{m_1})^2\times(\N^{m_2})^2$ is a finite set and $0\ne v_{\boldsymbol{p}\boldsymbol{q}\boldsymbol{r}\boldsymbol{s}}\in V$, define the degree of $g$ to be the maximal element of $E$, and denote it by $\deg(g)$. For zero element, we assume its degree is infinitesimal.
For $1\le l\le \max(m_1,m_2)$, define $P_l=\max\left\{p_l\mid (\boldsymbol{p},\boldsymbol{q},\boldsymbol{r},\boldsymbol{s})\in E\right\}$ and $R_l=\max\left\{r_l\mid (\boldsymbol{p},\boldsymbol{q},\boldsymbol{r},\boldsymbol{s})\in E\right\}$.
We use $e_l$ and $\epsilon_l$ to denote $(\delta_{1l},\ldots,\delta_{m_1l})\in\N^{m_1}$ and $(\delta_{1l},\ldots,\delta_{m_2l})\in\N^{m_2}$, respectively.
\subsection{Simplicity}
Let $X\in\left\{L,H,I,J\right\}$ and $p\in\Z$.
For a non-zero element
\[
g=\sum_{(\boldsymbol{p},\boldsymbol{q},\boldsymbol{r},\boldsymbol{s})\in E}\boldsymbol{s}^{\boldsymbol{p}}\boldsymbol{t}^{\boldsymbol{q}}\otimes\boldsymbol{x}^{\boldsymbol{r}}\boldsymbol{y}^{\boldsymbol{s}}\otimes v_{\boldsymbol{p}\boldsymbol{q}\boldsymbol{r}\boldsymbol{s}}\in\boldsymbol{T}(\boldsymbol{m},\boldsymbol{\lambda},\boldsymbol{\sigma},\boldsymbol{\eta},V),
\]
define a subspace
\begin{equation}
	\label{NX}
	N(X,g,p)=\mathrm{span}_{\C}\left\{g,X_ng\mid n\ge p\right\}.
\end{equation}
Then, the following lemma holds, whose proof is parallel to that of Proposition 3.2 in \cite{CY1}.
\begin{lemma}
	\label{syy}
	Suppose $\lambda_1,\lambda_2,\ldots,\lambda_{m_1+m_2}\in\C^*$ are pairwise distinct.
	 Take
	\[
	0\ne g=\sum_{(\boldsymbol{p},\boldsymbol{q},\boldsymbol{r},\boldsymbol{s})\in E}\boldsymbol{s}^{\boldsymbol{p}}\boldsymbol{t}^{\boldsymbol{q}}\otimes\boldsymbol{x}^{\boldsymbol{r}}\boldsymbol{y}^{\boldsymbol{s}}\otimes v_{\boldsymbol{p}\boldsymbol{q}\boldsymbol{r}\boldsymbol{s}}\in \boldsymbol{T}(\boldsymbol{m},\boldsymbol{\lambda},\boldsymbol{\sigma},\boldsymbol{\eta},V).
	\]
	Choose $s_1,s_2\in\N$ such that $L_mv_{\boldsymbol{p}\boldsymbol{q}\boldsymbol{r}\boldsymbol{s}}=H_nv_{\boldsymbol{p}\boldsymbol{q}\boldsymbol{r}\boldsymbol{s}}=0$ for $m\ge s_1$, $n\ge s_2$, and $(\boldsymbol{p},\boldsymbol{q},\boldsymbol{r},\boldsymbol{s})\in E$.
	Then, for $i\ge s_1$, $j\ge s_2$, and $1\le l\le \max(m_1,m_2)$, the following statements hold.
	\begin{align}
		\label{1}
		&\sum_{(\boldsymbol{p},\boldsymbol{q},\boldsymbol{r},\boldsymbol{s})\in E}\boldsymbol{s}^{\boldsymbol{p}+e_l}\boldsymbol{t}^{\boldsymbol{q}}\otimes\boldsymbol{x}^{\boldsymbol{r}}\boldsymbol{y}^{\boldsymbol{s}}\otimes v_{\boldsymbol{p}\boldsymbol{q}\boldsymbol{r}\boldsymbol{s}}\in N(L,g,i),\\
		\label{2}
		&\sum_{(\boldsymbol{p},\boldsymbol{q},\boldsymbol{r},\boldsymbol{s})\in E}\boldsymbol{s}^{\boldsymbol{p}}\boldsymbol{t}^{\boldsymbol{q}}\otimes\boldsymbol{x}^{\boldsymbol{r}+\epsilon_l}\boldsymbol{y}^{\boldsymbol{s}}\otimes v_{\boldsymbol{p}\boldsymbol{q}\boldsymbol{r}\boldsymbol{s}}\in N(L,g,i),\\
		\label{3}
		&\sum_{(\boldsymbol{p},\boldsymbol{q},\boldsymbol{r},\boldsymbol{s})\in E}\boldsymbol{s}^{\boldsymbol{p}}\boldsymbol{t}^{\boldsymbol{q}+e_l}\otimes\boldsymbol{x}^{\boldsymbol{r}}\boldsymbol{y}^{\boldsymbol{s}}\otimes v_{\boldsymbol{p}\boldsymbol{q}\boldsymbol{r}\boldsymbol{s}}\in N(H,g,j),\\
		\label{4}
		&\sum_{(\boldsymbol{p},\boldsymbol{q},\boldsymbol{r},\boldsymbol{s})\in E}\boldsymbol{s}^{\boldsymbol{p}}\boldsymbol{t}^{\boldsymbol{q}}\otimes\boldsymbol{x}^{\boldsymbol{r}}\boldsymbol{y}^{\boldsymbol{s}+\epsilon_l}\otimes v_{\boldsymbol{p}\boldsymbol{q}\boldsymbol{r}\boldsymbol{s}}\in N(H,g,j),\\
		\label{5}
		&\sum_{\substack{(\boldsymbol{p},\boldsymbol{q},\boldsymbol{r},\boldsymbol{s})\in E \\ p_l=P_l }}\boldsymbol{s}^{\boldsymbol{p}-P_le_l}\boldsymbol{t}^{\boldsymbol{q}+e_l}\otimes\boldsymbol{x}^{\boldsymbol{r}}\boldsymbol{y}^{\boldsymbol{s}}\otimes v_{\boldsymbol{p}\boldsymbol{q}\boldsymbol{r}\boldsymbol{s}}\in N(H,g,j),\\
		\label{6}
		&\sum_{\substack{(\boldsymbol{p},\boldsymbol{q},\boldsymbol{r},\boldsymbol{s})\in E \\ r_l=R_l}}\boldsymbol{s}^{\boldsymbol{p}}\boldsymbol{t}^{\boldsymbol{q}}\otimes\boldsymbol{x}^{\boldsymbol{r}-R_l\epsilon_l}\boldsymbol{y}^{\boldsymbol{s}+\epsilon_l}\otimes v_{\boldsymbol{p}\boldsymbol{q}\boldsymbol{r}\boldsymbol{s}}\in N(H,g,j).
	\end{align}
\end{lemma}
A direct result of \eqref{1}--\eqref{4} is the following assertion, whose proof is parallel to that of Proposition 3.3 in \cite{CY1}.
\begin{lemma}
	\label{generate}
	Suppose $\lambda_1,\lambda_2,\ldots,\lambda_{m_1+m_2}\in\C^*$ are pairwise distinct.
	Then the tensor product $\boldsymbol{T}(\boldsymbol{m},\boldsymbol{\lambda},\boldsymbol{\sigma},\boldsymbol{\eta},V)$ can be generated by $1\otimes\cdots\otimes1\otimes v$ for an arbitrary $0\ne v\in V$. 
\end{lemma}
Now, we are in a position to prove the \textbf{main result} of this subsection.
\begin{thm}
	\label{simple}
	The tensor product $\mathcal{G}$-module $\boldsymbol{T}(\boldsymbol{m},\boldsymbol{\lambda},\boldsymbol{\sigma},\boldsymbol{\eta},V)$ is simple if and only if any two of $\lambda_1,\lambda_2,\ldots,\lambda_{m_1+m_2}\in\C^*$ are distinct.
\end{thm}
\begin{proof}
	First, suppose $\lambda_1,\lambda_2,\ldots,\lambda_{m_1+m_2}\in\C^*$ are pairwise distinct and $N$ is a non-zero submodule.
	Take 
	\[
	0\ne g=\sum_{(\boldsymbol{p},\boldsymbol{q},\boldsymbol{r},\boldsymbol{s})\in E}\boldsymbol{s}^{\boldsymbol{p}}\boldsymbol{t}^{\boldsymbol{q}}\otimes\boldsymbol{x}^{\boldsymbol{r}}\boldsymbol{y}^{\boldsymbol{s}}\otimes v_{\boldsymbol{p}\boldsymbol{q}\boldsymbol{r}\boldsymbol{s}}\in N
	\]
	with the minimal degree $(\boldsymbol{i},\boldsymbol{j},\boldsymbol{k},\boldsymbol{l})$.
	Then \eqref{5} and \eqref{6} show that the entries of $\boldsymbol{i}$ and $\boldsymbol{k}$ are all zero.
	It follows that \[
	g=\sum_{(\boldsymbol{0},\boldsymbol{q},\boldsymbol{0},\boldsymbol{s})\in E}\boldsymbol{t}^{\boldsymbol{q}}\otimes\boldsymbol{y}^{\boldsymbol{s}}\otimes v_{\boldsymbol{0}\boldsymbol{q}\boldsymbol{0}\boldsymbol{s}}.
	\]
	If $\boldsymbol{j}\ne\boldsymbol{0}$, choose $s\in\N$ such that $I_nv_{\boldsymbol{p}\boldsymbol{q}\boldsymbol{r}\boldsymbol{s}}=0$ for $n\ge s$ and $(\boldsymbol{p},\boldsymbol{q},\boldsymbol{r},\boldsymbol{s})\in E$, then applying $I_n$ with $n\ge s$, we have
	\begin{align*}
		&I_n\left(\sum_{(\boldsymbol{0},\boldsymbol{q},\boldsymbol{0},\boldsymbol{s})\in E}\boldsymbol{t}^{\boldsymbol{q}}\otimes\boldsymbol{y}^{\boldsymbol{s}}\otimes v_{\boldsymbol{0}\boldsymbol{q}\boldsymbol{0}\boldsymbol{s}}\right)\\
		=&\sum_{k=1}^{m_1}\lambda_k^{n}\sigma_k\left(\sum_{(\boldsymbol{0},\boldsymbol{q},\boldsymbol{0},\boldsymbol{s})\in E} t_1^{q_1}\otimes\cdots\otimes (t_k-1)^{q_k}\otimes\cdots\otimes t_{m_1}^{q_{m_1}}\otimes\boldsymbol{y}^{\boldsymbol{s}}\otimes v_{\boldsymbol{0}\boldsymbol{q}\boldsymbol{0}\boldsymbol{s}}\right).
	\end{align*}
	Using the Vandermonde determinant and noting that $\sigma_k\ne 0$, we see that
	\begin{equation}
		\label{u_k}
	u_{k}=\sum_{(\boldsymbol{0},\boldsymbol{q},\boldsymbol{0},\boldsymbol{s})\in E} t_1^{q_1}\otimes\cdots\otimes (t_k-1)^{q_k}\otimes\cdots\otimes t_{m_1}^{q_{m_1}}\otimes\boldsymbol{y}^{\boldsymbol{s}}\otimes v_{\boldsymbol{0}\boldsymbol{q}\boldsymbol{0}\boldsymbol{s}}\in N
	\end{equation}
	for $1\le k\le m_1$.
	If $\boldsymbol{j}\ne\boldsymbol{0}$, let $i=\min\left\{k\mid j_k\ne 0\right\}$, then $g\ne u_i$ and  $\deg\left(g-u_i\right)\prec (\boldsymbol{0},\boldsymbol{j},\boldsymbol{0},\boldsymbol{l})$, which is a contradiction.
	Hence we have $\boldsymbol{j}=\boldsymbol{0}$.
	Consider the action of $J_n$, for $1\le l\le m_2$,
	similar arguments demonstrate that
	\begin{equation}
		\label{w_l}
	w_l=\sum_{(\boldsymbol{0},\boldsymbol{0},\boldsymbol{0},\boldsymbol{s})\in E}1^{\otimes m_1}\otimes y_1^{s_1}\otimes\cdots\otimes (y_l+1)^{s_l}\otimes\cdots\otimes y_{m_2}^{s_{m_2}}\otimes v_{\boldsymbol{0}\boldsymbol{0}\boldsymbol{0}\boldsymbol{s}}\in N,
	\end{equation}
	and hence we have $\boldsymbol{l}=\boldsymbol{0}$.
	Now, we see that $1\otimes\cdots\otimes1\otimes v$ for some $0\ne v\in V$, combined with Lemma \ref{generate}, we conclude that $N=\boldsymbol{T}(\boldsymbol{m},\boldsymbol{\lambda},\boldsymbol{\sigma},\boldsymbol{\eta},V)$.
	
	Conversely, assume that $\lambda_1,\lambda_2,\ldots,\lambda_{m_1+m_2}\in\C^*$ are not pairwise distinct.
	If $\lambda_i=\lambda_j$ for $1\le i\ne j\le m_1+m_2$, denote $\lambda_i$ by $\lambda$. 
	Then by Propositions 3.3 and 4.4 in \cite{CGZ2}, the following three tensor product $\mathcal{G}$-modules are reducible:
	\begin{align*}
		\Omega(\lambda,\sigma_i,\eta_i)\otimes \Omega(\lambda,\sigma_j,\eta_j),\quad
		\Omega(\lambda,\sigma_i,\eta_i)\otimes \Gamma(\lambda,\sigma_j,\eta_j),\quad\mbox{and }\quad
		\Gamma(\lambda,\sigma_i,\eta_i)\otimes \Gamma(\lambda,\sigma_j,\eta_j).
	\end{align*}
	Hence $\boldsymbol{T}(\boldsymbol{m},\boldsymbol{\lambda},\boldsymbol{\sigma},\boldsymbol{\eta},V)$ is a reducible $\mathcal{G}$-module.
\end{proof}
\subsection{Isomorphism Classes}
In this subsection, we shall determine the isomorphism classes of simple tensor product modules $\boldsymbol{T}(\boldsymbol{m},\boldsymbol{\lambda},\boldsymbol{\sigma},\boldsymbol{\eta},V)$.
According to Theorem \ref{simple}, throughout this subsection, we always assume $\lambda_1,\lambda_2,\ldots,\lambda_{m_1+m_2}\in\C^*$ are pairwise distinct.
Recall that in \eqref{NX}, a subspace $N(X,g,p)$ is introduced.

For $0\ne g\in \boldsymbol{T}(\boldsymbol{m},\boldsymbol{\lambda},\boldsymbol{\sigma},\boldsymbol{\eta},V)$, define
\begin{equation*}
	D_g=\lim_{p\to+\infty}\dim_{\C}\left(N(H,g,p)+N(I,g,p)+N(J,g,p)\right)
\end{equation*}
and
\begin{equation*}
	D_{\boldsymbol{T}}=\inf\left\{D_g\mid 0\ne g\in \boldsymbol{T}(\boldsymbol{m},\boldsymbol{\lambda},\boldsymbol{\sigma},\boldsymbol{\eta},V)\right\}.
\end{equation*}
\begin{lemma}
	\label{index}
	For any $0\ne g\in \boldsymbol{T}(\boldsymbol{m},\boldsymbol{\lambda},\boldsymbol{\sigma},\boldsymbol{\eta},V)$, we always have $D_g\ge m_1+m_2+1$, where the equality holds if and only if $g=1\otimes\cdots\otimes1\otimes v$ for some $0\ne v\in V$.
\end{lemma}
\begin{proof}
Suppose we have a non-zero element
\[
g=\sum_{(\boldsymbol{p},\boldsymbol{q},\boldsymbol{r},\boldsymbol{s})\in E}\boldsymbol{s}^{\boldsymbol{p}}\boldsymbol{t}^{\boldsymbol{q}}\otimes\boldsymbol{x}^{\boldsymbol{r}}\boldsymbol{y}^{\boldsymbol{s}}\otimes v_{\boldsymbol{p}\boldsymbol{q}\boldsymbol{r}\boldsymbol{s}}\in \boldsymbol{T}(\boldsymbol{m},\boldsymbol{\lambda},\boldsymbol{\sigma},\boldsymbol{\eta},V)
\]	
and $\deg(g)=(\boldsymbol{i},\boldsymbol{j},\boldsymbol{k},\boldsymbol{l})$.
Choose $s\in\N$ such that $H_nv_{\boldsymbol{p}\boldsymbol{q}\boldsymbol{r}\boldsymbol{s}}=I_nv_{\boldsymbol{p}\boldsymbol{q}\boldsymbol{r}\boldsymbol{s}}=J_nv_{\boldsymbol{p}\boldsymbol{q}\boldsymbol{r}\boldsymbol{s}}=0$ for all $n\ge s$ and $(\boldsymbol{p},\boldsymbol{q},\boldsymbol{r},\boldsymbol{s})\in E$.
Then by \eqref{3} and \eqref{4} in Lemma \ref{syy}, for $i\ge s$ and $1\le l\le \max(m_1,m_2)$, we see that
$$c_l=\sum_{(\boldsymbol{p},\boldsymbol{q},\boldsymbol{r},\boldsymbol{s})\in E}\boldsymbol{s}^{\boldsymbol{p}}\boldsymbol{t}^{\boldsymbol{q}+e_l}\otimes\boldsymbol{x}^{\boldsymbol{r}}\boldsymbol{y}^{\boldsymbol{s}}\otimes v_{\boldsymbol{p}\boldsymbol{q}\boldsymbol{r}\boldsymbol{s}}$$
and
$$d_l=\sum_{(\boldsymbol{p},\boldsymbol{q},\boldsymbol{r},\boldsymbol{s})\in E}\boldsymbol{s}^{\boldsymbol{p}}\boldsymbol{t}^{\boldsymbol{q}}\otimes\boldsymbol{x}^{\boldsymbol{r}}\boldsymbol{y}^{\boldsymbol{s}+\epsilon_l}\otimes v_{\boldsymbol{p}\boldsymbol{q}\boldsymbol{r}\boldsymbol{s}}$$
belong to $N(H,g,i)$.

If $\boldsymbol{i}\ne\boldsymbol{0}$, let $k=\min\left\{q\mid i_q\ne 0\right\}$, then by \eqref{5}, for $i\ge s$, we have
\[
a_k=\sum_{\substack{(\boldsymbol{p},\boldsymbol{q},\boldsymbol{r},\boldsymbol{s})\in E \\ p_k=P_k }}\boldsymbol{s}^{\boldsymbol{p}-P_ke_k}\boldsymbol{t}^{\boldsymbol{q}+e_k}\otimes\boldsymbol{x}^{\boldsymbol{r}}\boldsymbol{y}^{\boldsymbol{s}}\otimes v_{\boldsymbol{p}\boldsymbol{q}\boldsymbol{r}\boldsymbol{s}}\in N(H,g,i).
\]
Now, we have
\[
\deg(a_k)\prec\deg(g)\prec\deg(d_{m_2})\prec\cdots\prec\deg(d_{1})\prec\deg(c_{m_1})\prec\cdots\prec\deg(c_1),
\]
and hence the dimension of the subspace of $N(H,g,i)$ spanned by $$\left\{a_k,g,c_l,d_l\mid 1\le l\le \max(m_1,m_2)\right\}$$ is greater than $m_1+m_2+1$.
It follows that $D_g> m_1+m_2+1$.
Similarly, if $\boldsymbol{i}=\boldsymbol{0}$ but $\boldsymbol{k}\ne\boldsymbol{0}$, then we also have $D_g> m_1+m_2+1$.

Now, suppose $\boldsymbol{i}=\boldsymbol{0}$, $\boldsymbol{k}=\boldsymbol{0}$, and $\boldsymbol{j}\ne\boldsymbol{0}$, let $r=\min\left\{q\mid j_q\ne0\right\}$. 
Then we have
\[
g=\sum_{(\boldsymbol{0},\boldsymbol{q},\boldsymbol{0},\boldsymbol{s})\in E}1^{\otimes(r-1)}\otimes t_r^{q_r}\otimes\cdots\otimes t_{m_1}^{q_{m_1}}\otimes\boldsymbol{y}^{\boldsymbol{s}}\otimes v_{\boldsymbol{0}\boldsymbol{q}\boldsymbol{0}\boldsymbol{s}}.
\]
By \eqref{u_k}, for $i\ge s$, we have
\[
u_{r}=\sum_{(\boldsymbol{0},\boldsymbol{q},\boldsymbol{0},\boldsymbol{s})\in E} 1^{\otimes(r-1)}\otimes (t_r-1)^{q_r}\otimes\cdots\otimes t_{m_1}^{q_{m_1}}\otimes\boldsymbol{y}^{\boldsymbol{s}}\otimes v_{\boldsymbol{0}\boldsymbol{q}\boldsymbol{0}\boldsymbol{s}}\in N(I,g,i)
\]
and $g\ne u_r$.
Note that
\[
\deg(g-u_r)\prec\deg(g)\prec\deg(d_{m_2})\prec\cdots\prec\deg(d_{1})\prec\deg(c_{m_1})\prec\cdots\prec\deg(c_1),
\]
it follows that $D_g>m_1+m_2+1$.
Similarly, if $\boldsymbol{i}=\boldsymbol{j}=\boldsymbol{0}$, $\boldsymbol{k}=\boldsymbol{0}$ but $\boldsymbol{l}\ne\boldsymbol{0}$, we also have $D_g>m_1+m_2+1$.

Finally, if $(\boldsymbol{i},\boldsymbol{j},\boldsymbol{k},\boldsymbol{l})=(\boldsymbol{0},\boldsymbol{0},\boldsymbol{0},\boldsymbol{0})$, then for $i\ge s$, we have
\begin{align*}
	N(H,g,i)=\mathrm{span}_{\C}\left\{g,c_l,d_l\mid 1\le l\le \max(m_1,m_2)\right\},\quad N(I,g,i)=N(J,g,i)=\C g,
\end{align*}
and hence $D_g=m_1+m_2+1$.
\end{proof}

Now, suppose
\[
\boldsymbol{T}(\boldsymbol{n},\boldsymbol{\mu},\boldsymbol{\tau},\boldsymbol{\gamma},W)=\left(\bigotimes_{k=1}^{n_1}\Omega(\mu_k,\tau_k,\gamma_k)\right)\otimes\left(\bigotimes_{k=n_1+1}^{n_1+n_2}\Gamma(\mu_k,\tau_k,\gamma_k)\right)\otimes W
\]
is another simple tensor product $\mathcal{G}$-module, where $\boldsymbol{n}=(n_1,n_2)\in\N^2$, $\boldsymbol{\mu}=(\mu_1,\mu_2,\ldots,\mu_{n_1+n_2}),\,\boldsymbol{\tau}=(\tau_1,\tau_2,\ldots,\tau_{n_1+n_2})\in(\C^*)^{n_1+n_2}$, $\boldsymbol{\gamma}=(\gamma_1,\gamma_2,\ldots,\gamma_{n_1+n_2})\in\C^{n_1+n_2}$, and $W$ is a simple restricted $\mathcal{G}$-module.
As vector spaces, we have
\[
\Omega(\mu_k,\tau_k,\gamma_k)=\C[s_k,t_k]
\] 
for $1\le k\le n_1$ and \[
\Gamma(\mu_k,\tau_k,\gamma_k)=\C[x_{k-n_1},y_{k-n_1}]
\]
for $n_1+1\le k\le n_1+n_2$.

The \textbf{main result} of this subsection is as follows.
\begin{thm}
	\label{iso}
	The simple tensor product modules
	\[
	\boldsymbol{T}(\boldsymbol{m},\boldsymbol{\lambda},\boldsymbol{\sigma},\boldsymbol{\eta},V)\cong\boldsymbol{T}(\boldsymbol{n},\boldsymbol{\mu},\boldsymbol{\tau},\boldsymbol{\gamma},W)
	\]
	if and only if $\boldsymbol{m}=\boldsymbol{n}$, $V\cong W$ as $\mathcal{G}$-modules, and 
	\begin{align*}
	\left\{(\lambda_k,\sigma_k,\eta_k)\mid 1\le k\le m_1\right\}&=\left\{(\mu_k,\tau_k,\gamma_k)\mid 1\le k\le m_1\right\},\\
	\left\{(\lambda_k,\sigma_k,\eta_k)\mid m_1+1\le k\le m_1+m_2\right\}&=\left\{(\mu_k,\tau_k,\gamma_k)\mid m_1+1\le k\le m_1+m_2\right\}.
	\end{align*}
\end{thm}
\begin{proof}
	Suppose $\phi:\boldsymbol{T}(\boldsymbol{m},\boldsymbol{\lambda},\boldsymbol{\sigma},\boldsymbol{\eta},V)\to\boldsymbol{T}(\boldsymbol{n},\boldsymbol{\mu},\boldsymbol{\tau},\boldsymbol{\gamma},W)$ is an isomorphism of $\mathcal{G}$-modules.
	Then, by Lemma \ref{index}, we have $m_1+m_2=n_1+n_2$.
	Also, for any $0\ne v\in V$, by Lemma \ref{index}, there exists a non-zero $w\in W$ (which certainly depends only on $v$) such that
	\[
	\phi(1\otimes\cdots\otimes1\otimes v)=1\otimes\cdots\otimes1\otimes w.
	\]
	Choose $s\in\N$ such that $\mathcal{G}_nv=0$ and $\mathcal{G}_nw=0$ for $n\ge s$. 
	From the equation
	\[
	\phi\left(I_n(1\otimes\cdots\otimes1\otimes v)\right)=I_n(1\otimes\cdots\otimes1\otimes w),
	\]
	we deduce that
	\[
	\sum_{k=1}^{m_1}\lambda_k^{n}\sigma_k-\sum_{k=1}^{n_1}\mu_k^{n}\tau_k=0.
	\]
	Note that $\sigma_k,\tau_k\in\C^*$, it follows from the Vandermonde determinant that $m_1=n_1$ and $(\lambda_k,\sigma_k)=(\mu_k,\tau_k)$ for $1\le k\le m_1$ up to a permutation.
	Similarly, the action of $J_n$ ensures that $(\lambda_k,\sigma_k)=(\mu_k,\tau_k)$ for $m_1+1\le k\le m_1+m_2$ up to a permutation.
	By the equation
	\[
	\phi\left(H_n(1\otimes\cdots\otimes1\otimes v)\right)=H_n(1\otimes\cdots\otimes1\otimes w),
	\]
	we have
	\begin{equation}
		\label{tt}
	\phi(1\otimes\cdots\otimes t_k\otimes\cdots\otimes1\otimes1^{\otimes m_2}\otimes v)=1\otimes\cdots\otimes t_k\otimes\cdots\otimes1\otimes1^{\otimes m_2}\otimes w
	\end{equation}
	and
	\begin{equation}
		\label{yy}
	\phi(1^{\otimes m_1}\otimes1\otimes\cdots\otimes y_k\otimes\cdots\otimes1\otimes v)=1^{\otimes m_1}\otimes1\otimes\cdots\otimes y_k\otimes\cdots\otimes1\otimes w.
	\end{equation}
	for $1\le k\le \max(m_1,m_2)$.
	
	The equation
	\[
	\phi\left(L_n(1\otimes\cdots\otimes1\otimes v)\right)=L_n(1\otimes\cdots\otimes1\otimes w)
	\]
	implies the following summation
	\begin{equation}
		\label{huge}
		\begin{split}
			&\sum_{k=1}^{m_1}\lambda_k^n\left(\phi(1\otimes\cdots\otimes s_k\otimes\cdots\otimes1\otimes1^{\otimes m_2}\otimes v)-1\otimes\cdots\otimes s_k\otimes\cdots\otimes1\otimes1^{\otimes m_2}\otimes w\right)\\
			+&\sum_{k=m_1+1}^{m_1+m_2}\lambda_k^{n}\left(\phi(1^{\otimes m_1}\otimes1\otimes\cdots\otimes x_{k-m_1}\otimes\cdots\otimes1\otimes v)-1^{\otimes m_1}\otimes1\otimes\cdots\otimes x_{k-m_1}\otimes\cdots\otimes1\otimes w\right)\\
			+&\sum_{k=1}^{m_1}n\lambda_k^n\left(\phi(1\otimes\cdots\otimes(t_k+\eta_k)\otimes\cdots\otimes1\otimes1^{\otimes m_2}\otimes v)-1\otimes\cdots\otimes(t_k+\gamma_k)\otimes\cdots\otimes1\otimes1^{\otimes m_2}\otimes w\right)\\
			+&\sum_{k=m_1+1}^{m_1+m_2}n\lambda_k^{n}\left(\phi(1^{\otimes m_1}\otimes1\otimes\cdots\otimes(y_{k-m_1}+\eta_{k})\otimes\cdots\otimes1\otimes v)\right.\\
			&\quad\quad\quad\quad\quad\quad\quad\left.-1^{\otimes m_1}\otimes1\otimes\cdots\otimes(y_{k-m_1}+\gamma_k)\otimes\cdots\otimes1\otimes w\right)
		\end{split}
	\end{equation}
	is equal to zero.
	Thus, Lemma \ref{crucial} indicates 
	$$\phi(1\otimes\cdots\otimes(t_k+\eta_k)\otimes\cdots\otimes1\otimes1^{\otimes m_2}\otimes v)=1\otimes\cdots\otimes(t_k+\gamma_k)\otimes\cdots\otimes1\otimes1^{\otimes m_2}\otimes w$$
	for $1\le k\le m_1$ and
	$$\phi(1^{\otimes m_1}\otimes1\otimes\cdots\otimes(y_{k-m_1}+\eta_{k})\otimes\cdots\otimes1\otimes v)=1^{\otimes m_1}\otimes1\otimes\cdots\otimes(y_{k-m_1}+\gamma_k)\otimes\cdots\otimes1\otimes w$$
	for $m_1+1\le k\le m_1+m_2$.
	Thus, from Equations \eqref{tt} and \eqref{yy}, we deduce that $\eta_{k}=\gamma_{k}$ for $1\le k\le m_1+m_2$.
	Now, the summation in \eqref{huge} implies
	\begin{equation}
		\label{ss}
	\phi(1\otimes\cdots\otimes s_k\otimes\cdots\otimes1\otimes1^{\otimes m_2}\otimes v)=1\otimes\cdots\otimes s_k\otimes\cdots\otimes1\otimes1^{\otimes m_2}\otimes w
	\end{equation}
	and
	\begin{equation}
		\label{xx}
	\phi(1^{\otimes m_1}\otimes1\otimes\cdots\otimes x_{k}\otimes\cdots\otimes1\otimes v)=1^{\otimes m_1}\otimes1\otimes\cdots\otimes x_{k}\otimes\cdots\otimes1\otimes w
	\end{equation}
	for $1\le k\le \max(m_1,m_2)$.
	Combining Equations \eqref{tt}, \eqref{yy}, \eqref{ss} and \eqref{xx}, it is straightforward to verify the map that sends $v\in V$ to $w\in W$ is an isomorphism of $\mathcal{G}$-modules.
\end{proof}
We conclude this paper with the following remarks.
\begin{rmk}
	\begin{enumerate}[label=(\arabic*)]
	\item
	Evidently, the simple tensor product $\mathcal{G}$-module $\boldsymbol{T}(\boldsymbol{m},\boldsymbol{\lambda},\boldsymbol{\sigma},\boldsymbol{\eta},V)$ constructed in \eqref{tp} is neither a weight module nor a restricted module, which implies the simple modules constructed in this paper are non-isomorphic to that in \cite{A,AK,GG,CY,CYY,LYZ}.
	\item
	If $m_1+m_2=1$ and $V=\C$ is the trivial module, then $\boldsymbol{T}(\boldsymbol{m},\boldsymbol{\lambda},\boldsymbol{\sigma},\boldsymbol{\eta},V)$ coincides with the modules in \cite{CGZ1}.
	If $m_1+m_2=2$ and $V=\C$ is the trivial module, then $\boldsymbol{T}(\boldsymbol{m},\boldsymbol{\lambda},\boldsymbol{\sigma},\boldsymbol{\eta},V)$ coincides with the modules in \cite{CGZ2}.
	Except for these situations, to our knowledge, the simple module $\boldsymbol{T}(\boldsymbol{m},\boldsymbol{\lambda},\boldsymbol{\sigma},\boldsymbol{\eta},V)$ has not appeared in previous literature.
\end{enumerate}
\end{rmk}

\section*{Acknowledgements}
	This paper is partially supported by the National Key R\&D Program of China (2024YFA1013802), NSF of China (11931009, 12101152, 12161141001, 12171132, 12401030, and 12401036), Innovation Program for Quantum Science and Technology (2021ZD0302902), and the Fundamental Research Funds for the Central Universities of China. 

\begin{thebibliography}{99}
	\bibitem{A}
	N.~Aizawa, Some properties of planar Galilean conformal algebras, in: {\it Lie Theory and Its Applications in Physics}, Springer Proc. Math. Stat. {\bf 36}, Springer, Tokyo, 2013, 301-309.
	
	\bibitem{AK}
	N.~Aizawa and Y.~Kimura, Galilean conformal algebras in two spatial dimension, arXiv:1112.0634.
	
	\bibitem{BG}
	A. Bagchi and R. Gopakumar, Galilean conformal algebras and AdS/CFT, J. High Energy Phys. {\bf 2009}, no.~7, 037, 22 pp.

    \bibitem{CGZ1}
    J. Cheng, D. Gao and Z. Zeng, Modules over the planar Galilean conformal algebra arising from free modules of rank one, Pacific J. Math. {\bf 326} (2023), no.~2, 227-249.
    
    \bibitem{CGZ2}
    J. Cheng, D. Gao and Z. Zeng, Tensor product modules over the planar Galilean conformal algebra from free modules of rank one, arXiv:2506.14893.
    
    \bibitem{CGZ}
    H. Chen, X. Guo and K. Zhao, Tensor product weight modules over the Virasoro algebra, J. Lond. Math. Soc. (2) {\bf 88} (2013), no.~3, 829-844.
    
    \bibitem{CH}
    Q. Chen and Y. He, 2-local derivations on the planar Galilean conformal algebra, Internat. J. Math. {\bf 34} (2023), no.~5, Paper No. 2350023, 13 pp.

	\bibitem{CY1}
	Q. Chen and Y.~F. Yao, A new class of irreducible modules over the affine-Virasoro algebra of type $A_1$, J. Algebra {\bf 608} (2022), 553-572.

    \bibitem{CY}
	Q. Chen and Y.~F. Yao, Simple restricted modules for the universal central extension of the planar Galilean conformal algebra, J. Algebra {\bf 634} (2023), 698-721.
	
	\bibitem{CYY}
	Q. Chen, Y.~F. Yao and H. Yang, Whittaker modules for the planar Galilean conformal algebra and its central extension, Comm. Algebra {\bf 50} (2022), no.~12, 5042-5059.
	
	\bibitem{CS}
	L. Chi and J. Sun, Left-symmetric algebra structures on the planar Galilean conformal algebra, Algebra Colloq. {\bf 26} (2019), no.~2, 285-308.
	
	\bibitem{CSY}
	L. Chi, J. Sun and H. Yang, Biderivations and linear commuting maps on the planar Galilean conformal algebra, Linear Multilinear Algebra {\bf 66} (2018), no.~8, 1606-1618.

    \bibitem{GG} 
	D. Gao and Y. Gao, Representations of the planar Galilean conformal algebra, Comm. Math. Phys. {\bf 391} (2022), no.~1, 199-221.

	\bibitem{GLP}
    S.~L. Gao, D. Liu and Y.~F. Pei, Structure of the planar Galilean conformal algebra, Rep. Math. Phys. {\bf 78} (2016), no.~1, 107-122.
    
    \bibitem{GLZ}
    X. Guo, S. Li and Q. Zhu, Tensor product of generalized polynomial modules for the Virasoro algebra, J. Algebra {\bf 660} (2024), 564-587.

	\bibitem{JSL}
	L. Jiang, H. Shen and D. Liu, Local derivations on the planar Galilean conformal algebra, Contemporary Mathematics {\bf 6} (2025), no.~1, 1265-1278.


    
	\bibitem{LGW}
	X.~W. Liu, X. Guo and J. Wang, A new class of irreducible Virasoro modules from tensor product, J. Algebra {\bf 541} (2020), 324-344.
	
	\bibitem{LYZ}
	R. Lü, X. You and K. Zhao, $(d,\sigma)$-twisted affine-Virasoro superalgebras, arXiv:2507.00349.
    

    \bibitem{MZ}
    V. Mazorchuk and K. Zhao, Simple Virasoro modules which are locally finite over a positive part, Selecta Math. (N.S.) {\bf 20} (2014), no.~3, 839-854.

	\bibitem{OW}
	M. Ondrus and E.~B. Wiesner, Modules induced from polynomial subalgebras of the Virasoso algebra, J. Algebra {\bf 504} (2018), 54-84.
	
	\bibitem{TZ}
	X.~M. Tang and Y. Zhong, Biderivations of the planar Galilean conformal algebra and their applications, Linear Multilinear Algebra {\bf 67} (2019), no.~4, 649-659.
	
	\bibitem{TanZ}
	H. Tan and K. Zhao, Irreducible Virasoro modules from tensor products (II), J. Algebra {\bf 394} (2013), 357-373.
	
	\bibitem{TZhao}
	H. Tan and K. Zhao, Irreducible Virasoro modules from tensor products, Ark. Mat. {\bf 54} (2016), no.~1, 181-200.
    
	\end{thebibliography}
\end{document}